\documentclass[draft]{article}

\usepackage{amssymb}
\usepackage{amsmath}
\usepackage{indentfirst}

\newtheorem{theorem}{Theorem}[section]

\newtheorem{corollary}[theorem]{Corollary}
\newtheorem{lemma}[theorem]{Lemma}
\newenvironment{proof}[1][Proof]{\textbf{#1.}}{\ \rule{0.5em}{0.5em}}

\begin{document}	
\title{A New Method of Construction of Permutation Trinomials with Coefficients 1}
\author{Hua Guo$^{1,2}$, Shuo Wang$^3$, Hutao Song$^{1,4}$, Xiyong Zhang$^{5,*}$, Jianwei Liu$^1$ \\
{\scriptsize		
\begin{tabular}{ll}	
1 & School of Cyber Science and Technology, Beihang University, Beijing 100191, China \\
2 & State Key Laboratory of Cryptology, P.O. Box 5159, Beijing 100878, China \\
3 & School of Electronic and Information Engineering, Beihang University, Beijing 100191, China \\
4 & State Key Laboratory of Software Development Environment, Beihang University, Beijing 100191, China\\
5 & Beijing Institute of Satellite Information Engineering, Beijing 100086, China \\
\end{tabular}
}}
\date{}	
\maketitle
\vspace{-0.4cm}	
\let\thefootnote\relax
\footnote{*Corresponding author}
\footnote{Email addresses:}
\footnote{ hguo@buaa.edu.cn (Hua Guo)}
\footnote{nicolas\_wangshuo@buaa.edu.cn (Shuo Wang)}
\footnote{htsong@buaa.edu.cn (Hutao Song)}
\footnote{Xiyong.Zhang@hotmail.com (Xiyong Zhang)}
\footnote{liujianwei@buaa.edu.cn (Jianwei Liu)}
\begin{abstract}
Permutation polynomials over finite fields are an interesting and constantly active research subject of study for many years. They have important applications in areas of mathematics and engineering. In recent years, permutation binomials and permutation trinomials attract people's interests due to their simple algebraic forms. In this paper, by reversely using Tu's method for the characterization of permutation polynomials with exponents of Niho type, we propose a new method to construct permutation trinomials with coefficients 1. Moreover, we give the explicit compositional inverses of a class of permutation trinomials for a special case.
\end{abstract}

\textbf{Keywords:}
{\textit{Finite field, Permutation trinomial, Fractional reciprocal polynomial, Polar decomposition, Exponential sum}
	
\section{Introduction}
Let $ p $ be a prime number, $ n $ be a positive integer, $ q=p^{n} $, $ \mathbb{F}_{q} $ be a finite field with $ q $ elements, and $\mathbb{F}_q^\times$ be its multiplicative group. Let  $ f(x) $ be a polynomial over $ \mathbb{F}_{q} $. If $ f : \alpha \mapsto f(\alpha) $ is a permutation of $ \mathbb{F}_{q} $, $ f(x) $ is called a permutation polynomial (abbreviated as PP) of $ \mathbb{F}_{q} $. The algebraic closure of $\mathbb{F}_q$ is denoted by $\overline{\mathbb{F}}_q$. For integer $s,d>0$ such that $sd=q-1$, let $\mu_d=\{x\in \overline{\mathbb{F}}_q:x^d=1\}$ denote the unique cyclic subgroup of $\mathbb{F}_q^\times$ of order $d$. Especially when $q=2^{2m}$ for some positive integer $m$ and $d=2^m+1$, $\mu_d$ is written as $U$ in this paper. Permutation polynomials have wide applications in various areas of mathematics and engineering, such as coding theory \cite{ref01}, cryptography \cite{ref03,ref02} and combinatorial designs \cite{ref04}. Permutation binomials and trinomials with coefficients 1 have the simplest algebraic form. These types of PPs attract people's interests in recent years due to their important applications in cryptography (such as multivariate public key cryptography), integer sequences, and finite geometry, etc.
	
In general by Lagrange interpolation it is not difficult to construct a random PP for a given finite field. However it is difficult to find PPs with simple or nice algebraic appearance. Only a few classes of permutation binomials and trinomials are known by now. In 2009, Akbary, Ghioca and Wang proposed a useful criterion referred as AGW criterion to decide whether a map is a permutation function over a finite set  in \cite{ref1}. In fact it is a generalization of numerous earlier results, one of which is the following theorem that has been used in many previous constructions of PPs.
	
\begin{theorem}\label{Th1.1}\cite{ref2}
Let $ q $ be a prime power number, $r,d $ be positive integers, $ d \mid (q-1) $, $ h \in \mathbb{F}_{q}[x] $, then  $ f(x): = x^{r}h(x^{(q-1)/d}) $ is a permutation polynomial over $ \mathbb{F}_{q} $ if and only if
		
(1) $\gcd(r,(q-1)/d) = 1 $, and
		
(2) $x^{r}h(x)^{(q-1)/d} $ is a permutation polynomial over $ \mu_{d} $.
\end{theorem}
	
Theorem \ref{Th1.1} shows that polynomials of the form $x^{r}h(x^{(q-1)/d})$ have a close connection with the $d$-order subgroup $\mu_d$ of $ \mathbb{F}_{q}^\times $. Indeed there have been many construction results by using the method of constructing PPs over $\mu_d$ to obtain PPs over the original finite field.
	
All permutation binomials $x^m+ax^n\in \mathbb{F}_q[x]$ have been characterized by Zieve \cite{ref2}. The necessary and sufficient condition for  $ x^{r}(x^{(q-1)/l}+a) $ being a permutation polynomial is given by the following Theorem.
	
\begin{theorem}\label{Th1.2}\cite{ref2}
Let $ q $ be a prime power number, $r,d $ be positive integers, $ d \mid (q-1) $, $a\in \mathbb{F}_{q}^\times$, and $\eta + \frac{a}{\eta} \in \mu_{(q-1)/d}$ for all $\eta \in \mu_{2d}$,  then  $ x^{r}(x^{(q-1)/d}+a)$ is a permutation polynomial over $ \mathbb{F}_{q} $ if and only if $-a \notin \mu_{d} $, $ \gcd (r,(q-1)/d)=1$, and $ \gcd (2d,2r+(q-1)/d) \leq 2$.
\end{theorem}
	
Theorem \ref{Th1.2} indicates that there are no permutation binomials with both nonzero coefficients equal to 1 over finite fields with even characteristic, while permutation trinomials with coefficients 1 exist \cite{ref14,ref9,ref8}. This motivates us to find a new method to construct permutation trinomials with coefficients 1.
	
We list some permutation trinomials with coefficients 1 constructed in recent years in Table \ref{tab:1}. The main construction methods include Hermite's criterion, Dobbertin's multivariate approach(\cite{ref5}), discussions of the number of solutions of special equations (\cite{ref7}), polar decomposition (\cite{ref7}),  etc. It should be noted that in many cases proving the permutation properties of PPs constructed by the method of polar decomposition or Theorem \ref{Th1.1} is equivalent to proving that the corresponding fractional polynomials permute $\mu_{\sqrt{q}+1}$. Then is it possible to obtain PPs of the original finite field by reversely using Tu's polar decomposition method by exponential sum? It seems that this idea has not appeared previously to our knowledge. The bulk of our effort in this paper is placed in this direction, and the main contribution is to present a novel construction method of PPs by using the above idea.
	
As for fractional permutation polynomials over $U$ which play important roles in the constructions of PPs by Theorem 1.1, Li et al. \cite{ref12} proved the permutation properties of $ \frac{x^{3}+x^{2}+1}{x^{3}+x+1} $ and $ \frac{x^{7}+x^{5}+1}{x^{7}+x^{2}+1} $ by determining the solutions of quadratic and quartic equations respectively. Later Li et al. \cite{ref14} proved the permutation properties of  $\frac{x^{2^{k}+1}+x^{2^{k}}+1}{x^{2^{k}+1}+x+1} $ over $U$ by rewriting it as $ x^{2^{k}}=\frac{ax+b}{cx+d} $. Main fractional permutation polynomials of $U$ over a finite field with even characteristic up to now are listed in Table \ref{tab:2}.
	
\begin{table}[h]
\renewcommand\arraystretch{2}
\caption{Some permutation trinomials with coefficients 1}
\label{tab:1}
\begin{center}
\begin{tabular}{c|c|c}
\hline
\textbf{Permutation polynomial of $ F_{2^{2m}} $}& 
\textit{\textbf{m}}& 
\textbf{Reference}\\ 
\hline	
$ x+x^{3}+x^{2^{(m+1)/2}+2} $ & odd & \cite{ref5} \\
$ x^{3\times 2^{(m+1)/2}+4}+x^{2^{(m+1)/2}+2}+x^{2^{(m+1)/2}} $ & odd & \cite{ref6} \\
$ x+x^{2^{(m+1)/2}-1}+x^{2^{m}-2^{(m+1)/2}+1} $ & odd & \cite{ref7} \\
$ x+x^{3}+x^{2^{m}-2^{(m+3)/2}+2} $ & odd  & \cite{ref7} \\
$ x+x^{2^{(m+1)/2}-1}+x^{2^{m}-2^{m/2}+1} $ & odd & \cite{ref7} \\
$ x^{2^{m/2}+4}+x^{2^{m/2+1}+3}+x^{2^{m/2+2}+1} $ & $ m\equiv 2(\mod 4) $ & \cite{ref4} \\
$ x^{5}+x^{2^{m/2}+4}+x^{5 \times 2^{m/2}} $ & $ m\equiv 4(\mod 8) $ & \cite{ref9},Theorem \ref{theorem 3.1} \\
$ x^{9}+x^{8+7 \times 2^{m}}+x^{9 \times 2^{m}} $ & $ m\equiv 2(\mod 4) $ & Theorem \ref{theorem 3.1} \\ 
\hline
\end{tabular}
\end{center}
\end{table}
	
\begin{table}[h]
\renewcommand\arraystretch{2}
\caption{Some fractional permutation polynomials over $ U $}
\label{tab:2}
\begin{center}
\begin{tabular}{c|c|c}
\hline
\textbf{Fraction polynomial}& 
\textbf{Condition}& 
\textbf{Reference}\\
\hline  
$ \frac{x^{3}+x^{2}+1}{x^{3}+x+1} $ & & \cite{ref7} \\
$ \frac{x(x^{3}+x^{2}+1)}{x^{3}+x+1} $ & $ \gcd(3,m)=1 $ & \cite{ref7} \\
$ \frac{x^{5}+x^{4}+1}{x^{5}+x+1} $ & $ m $ is odd & \cite{ref12,ref14} \\
$ \frac{x^{4}+x^{3}+1}{x(x^{4}+x+1)} $ & $ m $ is odd & \cite{ref14} \\
$ \frac{x^{7}+x^{5}+1}{x^{7}+x^{2}+1} $ & $ m $ is odd & \cite{ref14} \\
$ \frac{x(x^{5}+x+1)}{x^{5}+x^{4}+1} $ & $ m=2,4(\mod 6) $ & \cite{ref12} \\
$ \frac{x^{6}+x^{4}+1}{x(x^{6}+x^{2}+1)} $ & $ \gcd(3,m)=1 $ & \cite{ref12} \\
$ \frac{x^{2^{k}+1}+x^{2^{k}}+1}{x^{2^{k}+1}+x+1} $ & $ \gcd(2^{k}-1,2^{m}+1)=1 $ & \cite{ref14} \\
$ \frac{x^{2^{k}+1}+x^{2^{k}}+1}{x^{2^{k}+1}+x+1} $ & $ \gcd(2^{k}+1,2^{m}+1)=1 $ & \cite{ref14} \\  
\hline
\end{tabular}
\end{center}
\end{table}
	
This paper is organized as follows. In Section 2, we introduce some basic notations and Lemmas. In Section 3, we propose a new method to construct permutation trinomials with coefficients 1, which is the main result of this paper. Then in Subsection 3.2, we give the explicit compositional inverses of a class of permutation trinomials for a special case. Some concluding remarks are given in Section 4.
	
\section{Preliminaries}
	
We recall Tu's method in this section, which gave a criterion for permutation polynomials in terms of additive characters of the underlying finite fields. Let $Tr^n_1(\cdot)$ be the absolute trace function from $\mathbb{F}_{2^n}$ to $\mathbb{F}_2$, i.e. $\forall x\in \mathbb{F}_{2^n}, Tr^n_1(x)=x+x^2+\cdots+x^{2^{n-1}}$. The following Lemma characterizes permutation polynomials by using additive characters.
	
\begin{lemma}\label{lemma 2.1}\cite{ref8}
A mapping $ g: \mathbb{F}_{2^{n}} \rightarrow \mathbb{F}_{2^{n}} $ is a permutation polynomial over $ \mathbb{F}_{2^{n}} $ if and only if for every nonzero $ \gamma \in \mathbb{F}_{2^{n}} $,
\begin{equation*}\label{1}
\sum_{x \in \mathbb{F}_{2^{n}}} (-1)^{Tr_{1}^{n}(\gamma g(x))}=0.
\end{equation*}
If $ f(x) = \sum\limits_{i=1}^{t} u_{i}x^{d_{i}} $ with $u_{1} = 1$ and $\gcd(d_{1}, 2^{n}-1) = 1 $, then
\begin{equation*}\label{2}
\begin{split}
\sum_{x \in \mathbb{F}_{2^{n}}} (-1)^{Tr_{1}^{n}(\gamma f(x))}
&=\sum_{x \in \mathbb{F}_{2^{n}}} (-1)^{Tr_{1}^{n}(\gamma (\sum\limits_{i=1}^{t} u_{i}x^{d_{i}}))}\\
&=\sum_{x \in \mathbb{F}_{2^{n}}} (-1)^{Tr_{1}^{n}((\delta x)^{d_{1}}+\sum\limits_{i=2}^{t} u_{i}\delta^{d_{1}-d_{i}}(\delta x)^{d_{i}})}\\
&=\sum_{x \in \mathbb{F}_{2^{n}}} (-1)^{Tr_{1}^{n}(x^{d_{1}}+\sum\limits_{i=2}^{t} u_{i}\delta^{d_{1}-d_{i}}x^{d_{i}})}.
\end{split}
\end{equation*}
\end{lemma}
	
Let $ m $ be a positive integer, $ n=2m $. For $ i=1,2,...,t $, $ d_{i}=s_{i}(2^{m}-1)+e $, then the computation of the above exponential sum can be converted to equation problem over $U$.
	
\begin{corollary}\label{cor2.2}\cite{ref8}
Let $ f(x) = x^{d_{1}} + \sum\limits_{i=2}^{t} u_{i}x^{d_{i}} $,  $ \gcd(d_{1},2^{n}-1) = 1 $, and  $ d_{i}=s_{i}(2^{m}-1)+e $ for every $i\in [1,t]$. Then the polynomial $ f $ is a permutation polynomial over $ \mathbb{F}_{2^{n}} $ if and only if $ \forall \delta \in \mathbb{F}_{2^{n}} $,
\begin{equation*}\label{tu2}
\begin{array}{ll}
&\sum\limits_{x \in \mathbb{F}_{2^{n}}} (-1)^{Tr_{1}^{n}( x^{d_{1}}+\sum\limits_{i=2}^{t} u_{i}\delta^{d_{1}-d_{i}}x^{d_{i}})}\\
&=(N(u_2,\cdots,u_t)-1)\cdot 2^m\\
&=0.
\end{array}
\end{equation*}
where $N(u_2,\cdots,u_t)$ is the number of $\lambda's$ in $U$ such that
\begin{equation}\label{tu-equation}
\lambda^{d_1}+\sum\limits_{i=2}^{t} u_{i}\delta^{d_{1}-d_{i}}\lambda^{d_{i}}+(\lambda^{d_1}+\sum\limits_{i=2}^{t} u_{i}\delta^{d_{1}-d_{i}}\lambda^{d_{i}})^{2^m}=0.
\end{equation}
\end{corollary}
	
As already mentioned in Section 1, there is a close connection between the fractional PPs over subgroup of order $d$ and the permutations of form $ x^{r}h(x^{(q-1)/d}) $ over $\mathbb{F}_q$. Thus constructing corresponding fractional PPs over $\mu_d$ is the key point for constructing certain permutations with the above form. The following Theorem gives a general fractional PPs over $U$ which would be needed later in this paper.
	
\begin{theorem}\label{theorem 2.5}\cite{ref14}
Let $ U $ be the $ (2^m+1) $-order multiplicative subgroup of the finite field $ \mathbb{F}_{2^{{2m}}}$. If $\gcd(2^{k}-1,2^{m}+1)=1$, the reciprocal polynomial $ \frac{x^{2^k+1}+x^{2^{k}}+1}{x^{2^k+1}+x+1}$ is a permutation fractional polynomial of $U$.
\end{theorem}
	
\section{A new method of construction of permutation trinomials with coefficients 1}
	
In this section, we firstly give a new construction method in Subsection 3.1. In order to illustrate how to use the new method, we give the construction process of a class of trinomials, and the explicit compositional inverses of this class of permutation trinomials for a special case are given in Subsection 3.2.
	
\subsection{New method}
	
In \cite{ref8}, Tu et al. proposed a criterion for permutation polynomials by using polar decomposition. They characterized the permutation properties of polynomials by calculating Walsh spectra of the corresponding Boolean functions in terms of additive characters of finite fields. Specifically, the study of permutation behavior of polynomials $\sum_{i=1}^t u_ix^{d_i}$ satisfying $d_i\equiv d_j(\mod 2^{n/2-1})$  for each pair of exponents $(d_i,d_j)$ is converted to the study of permutation behavior of polynomials $g(x)+g(x)^{2^m}$ over $U$. In this subsection, we are inspired by the research of \cite{ref8}, and reversely use the polar decomposition method to deduce a large class of permutation trinomials with coefficients 1.
	
\begin{theorem}\label{theorem 3.1}
Let $n=2m,m>0$, $ i,j\in \mathbb{N}^{+}$, $ i \neq j $, $u\in \mathbb{Z}$, $J=2^j,I=2^i$. Suppose $ \gcd(d_1,2^{2m}-1)=1 $ and $ \gcd(2^{i}-2^{j},2^{m}+1)=1 $. And
\begin{equation*}
\left\{
\begin{array}{ll}
d_1&=2^{i-1}-2^{j-1}+u \times (2^{m}+1),\\
d_2&=2^{i-1}+2^{j-1}+(u-2^{j-1}) \times (2^{m}+1),\\
d_3&=-(2^{i-1}+2^{j-1})+(u+2^{i-1}) \times (2^{m}+1).
\end{array}
\right.
\end{equation*}
where $\frac{1}{J-I}(\mod  2^m+1)\equiv t$, $0< t< 2^m+1$. Then  $$ x^{d_1}+x^{d_2}+x^{d_3}$$ is a permutation polynomial over $ \mathbb{F}_{2^{n}} $.
\end{theorem}
\begin{proof}
Write $k: =j-i$. Since $ \gcd(2^{i}-2^{j},2^{m}+1)=\gcd(2^k-1,2^m+1)=1 $, by Theorem \ref{theorem 2.5}, $$ f_{k}(x)=\frac{x^{K+1}+x^{K}+1}{x^{K+1}+x+1}, K=2^{k}, $$ permutes $ U $.
		
Note $ \gcd(I,2^{k}-1)=1 $, we have $ g(x)=x^{I} $ is a permutation polynomial over $ U $, and $ g(f_{k}(x))=(\frac{x^{K+1}+x^{K}+1}{x^{K+1}+x+1})^{I} $ is also a permutation polynomial of $ U $. Therefore, for all $\beta\in U$, $$ (x^{K+1}+x^{K}+1)^{I}=\beta(x^{K+1}+x+1)^{I}$$ has only one solution in $ U $. That is,
\begin{equation}\label{Th3.5-1}
(1+\beta)x^{I+J}+x^{J}+\beta x^{I}+1+\beta=0, J=KI,
\end{equation}
has only one solution in $ U $.
		
For all $t\in \mathbb{Z}$, let $ x=\beta^{t}y $. Substituting it into (\ref{Th3.5-1}), we get
\begin{equation*}\label{16}
(\beta^{t(I+J)+1}+\beta^{t(I+J)})y^{I+J}+\beta^{tJ} y^{J}+\beta^{tI+1} y^{I}+1+\beta=0.
\end{equation*}
		
If $ \gcd(J-I,2^{m}+1)=1 $, there exists $t\equiv\frac{1}{J-I}(\mod 2^m+1)$, $0<t<2^m+1$, such that $tJ=tI+1+b(2^m+1)=L$, where $b\in \mathbb{Z}$. Note $ \beta^{L} \in U $. Hence
\begin{equation}\label{Th3.5-2}
(\beta^{2L}+\beta^{2L-1})y^{I+J}+\beta^{L} y^{J}+\beta^{L} y^{I}+1+\beta=0.
\end{equation}
		
Write $ L_{1}:=\frac{I+J}{2}, L_{2}:=\frac{I-J}{2} $. Dividing both sides of equation (\ref{Th3.5-2}) by $ \beta^{L}y^{L_{1}} $, we have
\begin{equation*}\label{18}
(\beta^{L}+\beta^{L-1})y^{L_{1}}+y^{L_{2}}+ y^{-L_{2}}+(\beta^{-L}+\beta^{-(L-1)})y^{-L_{1}}=0,
\end{equation*}
i.e.,
\begin{equation}\label{tu-shape}
\beta^{L}y^{L_{1}}+y^{L_{2}}+ \beta^{-(L-1)}y^{-L_{1}}+(\beta^{L}y^{L_{1}}+y^{L_{2}}+ \beta^{-(L-1)}y^{L_{1}})^{2^m}=0.
\end{equation}
		
According to Lemma \ref{tu2}, the three-term polynomial obtained by the method of polar decomposition should be
\begin{equation}\label{Th3.5-3}
y^{L_{2}}+\beta^{L}y^{L_{1}}+ \beta^{-(L-1)}y^{-L_{1}}.
\end{equation}
		
On the other hand, for the three-term polynomial with coefficients 1, the simplified polynomial form deduced from the proof of permutation polynomial over the underlying finite field by Tu's method in \cite{ref8} is
\begin{equation}\label{Th3.5-4}
y^{d_{1}}+\delta^{d_{1}-d_{2}}y^{d_{2}}+\delta^{d_{1}-d_{3}}y^{d_{3}}, y\in U, \delta \in \mathbb{F}_{2^{n}}.
\end{equation}
		
Let $ Q=2^{m}+1 $. According to (\ref{Th3.5-3}) and (\ref{Th3.5-4}), we can assume that
\begin{equation}\label{Th3.5-5}
\left\{\begin{matrix} L_{2} \equiv d_{1}\ (\mod Q), \\ L_{1} \equiv d_{2}\ (\mod Q), \\ -L_{1} \equiv d_{3}\ (\mod Q). \end{matrix}\right.
\end{equation}
		
Furthermore we can assume $ \exists u_{1},u_{2},u_{3},u_{4},u_{5}, a \in \mathbb{Z} $, such that $ \forall y \in U $, $ \delta \in \mathbb{F}_{2^{n}} $,
\begin{equation}\label{Th3.5-6}
\begin{array}{ll}
&y^{L_{2}}+\beta^{L}y^{L_{1}}+ \beta^{-(L-1)}y^{-L_{1}}\\
=&y^{L_{2}+u_{1}Q}+\delta^{(Q-2)aL+u_{4}Q(Q-2)}y^{L_{1}+u_{2}Q}+\delta^{(Q-2)a(1-L)+u_{5}Q(Q-2)}y^{-L_{1}+u_{3}Q}.
\end{array}
\end{equation}
		
By (\ref{Th3.5-4}) and the polynomial form of (\ref{Th3.5-6}), we have
\begin{equation}\label{24}
\left\{\begin{matrix}  d_{1}=L_{2}+u_{1}Q, \\ d_{1}-d_{2}=(Q-2)aL+u_{4}Q(Q-2), \\ d_{1}-d_{3}=(Q-2)a(1-L)+u_{5}Q(Q-2).  \end{matrix}\right.
\end{equation}
		
Thus,	
\begin{equation}\label{amend 1}
\left\{\begin{matrix}  d_{2}=L_{2}+u_{1}Q-(Q-2)aL-u_{4}Q(Q-2), \\ d_{3}=L_{2}+u_{1}Q+(Q-2)a(L-1)-u_{5}Q(Q-2).  \end{matrix}\right.
\end{equation}
		
By (\ref{Th3.5-5}) and (\ref{amend 1}), we have
\begin{equation}\label{amend 2}
\left\{\begin{matrix}  d_{2} \equiv L_{2}+2aL \equiv L_{1}\ (\mod Q), \\ d_{3} \equiv L_{2}-2a(L-1) \equiv -L_{1}\ (\mod Q).  \end{matrix}\right.
\end{equation}
		
Since $ L_{1}=\frac{I+J}{2}$, $L_{2}=\frac{I-J}{2} $, $ t\equiv\frac{1}{J-I}(\mod 2^m+1) $, and $L=tJ=tI+1+b(2^m+1)$, by solving the equations (\ref{amend 2}), we get
\begin{equation}\label{amend 3}
a \equiv \frac{1}{2t}\ (\mod Q). 
\end{equation} 
		
Substitute $ a=\frac{1}{2t}+u_{6}Q $ ($ u_{6}\in \mathbb{Z} $) into the equations (\ref{amend 1}),
\begin{equation}\label{amend 4}
\left\{\begin{matrix}  d_{2}=L_{2}+u_{1}Q-(Q-2)\frac{J}{2}-(u_{4}+u_{6}tJ)Q(Q-2), \\ d_{3}=L_{2}+u_{1}Q+(Q-2)\frac{I}{2}+(\frac{b}{2t}+u_{6}tI+u_{6}bQ-u_{5})Q(Q-2).  \end{matrix}\right.
\end{equation}
		
Substitute $L_{2}$ into $d_{1}$, $d_{2}$ and $d_{3}$, and replace $u_{1}$ by $u$, we have
\begin{equation}\label{amend 5}
\left\{\begin{matrix}  d_{1} \equiv \frac{I-J}{2}+uQ\ (\mod Q(Q-2)), \\ d_{2} \equiv \frac{I-J}{2}+uQ-(Q-2)\frac{J}{2}\ (\mod Q(Q-2)), \\ d_{3} \equiv \frac{I-J}{2}+uQ+(Q-2)\frac{I}{2}\ (\mod Q(Q-2)).  \end{matrix}\right.
\end{equation}
		
Thus, we deduce that the trinomial is
\begin{equation}\label{tu-backward}
x^{\frac{I-J}{2}+uQ}+x^{\frac{I-J}{2}+uQ-(Q-2)\frac{J}{2}}+x^{\frac{I-J}{2}+uQ+(Q-2)\frac{I}{2}}.
\end{equation}
It can be seen that the equation (\ref{tu-shape}) deduced from polar decomposition of (\ref{tu-backward}) has one and only one solution of $U$ for every $\beta\in U$.
		
Therefore,
\begin{equation*}
\left\{
\begin{array}{ll}
d_1&=2^{i-1}-2^{j-1}+u \times (2^{m}+1),\\
d_2&=2^{i-1}+2^{j-1}+(u-2^{j-1}) \times (2^{m}+1),\\
d_3&=-(2^{i-1}+2^{j-1})+(u+2^{i-1}) \times (2^{m}+1).
\end{array}
\right.
\end{equation*}
		
By Corollary 2.2, we have proved that for all $i,j \in \mathbb{N}^{+}$, $ i \neq j $, $ u\in \mathbb{Z} $ , if $ \gcd(2^{i-1}-2^{j-1}+u \times (2^{m}+1),2^{2m}-1)=1 $ and $ \gcd(2^{i}-2^{j},2^{m}+1)=1 $,  the polynomial $$x^{d_1}+x^{d_2}+x^{d_3}$$ is a permutation polynomial over $ \mathbb{F}_{2^{n}} $. The proof is complete. 
\end{proof}
	
\subsection{Compositional inverses of permutation trinomials in Subsection 3.1 for $ i=j+m-1 $}
Compositional inverses of permutation polynomials have important applications in cryptography. For instance, in multivariate public key cryptography, permutation polynomials with coefficients 1 can be used as central map due to their simple form, thus, their explicit compositional inverses will be needed in the decryption process.

For a polynomial $ f(x) $ over $ \mathbb{F}_{q} $, $ f^{-1}(x) $ is the compositional inverse of $ f(x) $ if $ f(f^{-1}(x)) \equiv f^{-1}(f(x)) \equiv x \pmod {x^{q}-x} $. In this subsection, we give the explicit compositional inverses of the permutation trinomials in Subsection 3.1 for a special case $ i=j+m-1 $.
	
\begin{lemma}\label{lemma 3.1}\cite{ref15}
Let $ q $ be a prime power and $ f(x)=x^{r}h(x^{s}) \in \mathbb{F}_{q}[x] $ be a permutation polynomial, where $ s | (q-1) $, $ \gcd(r,q-1)=1 $ and $ h(0) \neq 0 $. Let $ r' $ be an integer which satisfies $ rr' \equiv 1 \pmod {q-1} $, $ \alpha(x)=x^{q-s} $, $ \beta(x)=x^{s} $ and $ l(x) $ be the compositional inverse of $ g(x)=x^{r}h(x)^{s} $ over $ \mu_{(q-1)/s} $, then $$ f^{-1}(x)=(\alpha(x)h(l(\beta(x)))^{s-1})^{r'}l(\beta(x)) $$ is the compositional inverse of $ f(x) $ over $ \mathbb{F}_{q} $.
\end{lemma}
	
The permutation trinomials over $ \mathbb{F}_{2^{2m}} $ constructed by Theorem {\rm\ref{theorem 3.1}} can be written as $ f(x)=x^{d_1}(1+(x^{2^m-1})^{-2^{j-1}}+(x^{2^m-1})^{2^{i-1}}) $, thus the corresponding polynomials over $ U $ are $ g(x)=x^{d_1}(1+x^{-2^{j-1}}+x^{2^{i-1}})^{2^m-1} $. According to Lemma \ref{lemma 3.1}, the key of finding explicit compositional inverses of $ f(x) $ over $ \mathbb{F}_{2^{2m}} $ is the compositional inverses of $ g(x) $ over $ U $.
	
\begin{theorem}\label{theorem 3.4}
Let $ i=j+m-1 $, the compositional inverses of $ g(x)=x^{d_1}(1+x^{-2^{j-1}}+x^{2^{i-1}})^{2^m-1} $ are $ g^{-1}(x)=x^{r_{1}r_{2}} $, where $ r_1 \cdot 2^{m-1} \equiv 1 \pmod {2^m+1} $ and $ r_2 \cdot 2^{j-1} \equiv 1 \pmod {2^m+1} $.
\end{theorem}
	
\begin{proof}
$ g(x) $ can be written as:$$ g(x)=\frac{x^{2^{i-1}-2^{j-1}}+x^{2^{i-1}}+x^{-2^{j-1}}}{1+x^{-2^{j-1}}+x^{2^{i-1}}}, $$
thus,
\begin{equation*}
\begin{split}
g(g^{-1}(x))&=\frac{x^{r_1r_22^{j-1}2^{m-1}-r_1r_22^{j-1}}+x^{r_1r_22^{j-1}2^{m-1}}+x^{-r_1r_22^{j-1}}}{1+x^{-r_1r_22^{j-1}}+x^{r_1r_22^{j-1}2^{m-1}}} \\
&=\frac{x^{1-r_1}+x+x^{-r_1}}{1+x^{-r_1}+x} \\
&=\frac{x(x^{-r_1}+1+x^{-r_1-1})}{x^{-r_1}+1+x}.
\end{split}
\end{equation*}
		
Since $ (r_1+2)2^{m-1} \equiv r_12^{m-1}+2^m \equiv 2^m+1 \equiv 0 \pmod{2^m+1} $ and $ \gcd(2^{m-1},2^m+1)=1 $, $ (r_1+2) \equiv 0 \pmod{2^m+1} $, that is, $ (-r_1-1) \equiv 1 \pmod {2^m+1} $. Therefore, $ g(g^{-1}(x))=x $. The proof is complete. 
\end{proof}
	
According to Lemma \ref{lemma 3.1} and Theorem \ref{theorem 3.4}, we obtain the next Theorem easily.
	
\begin{theorem}\label{theorem 3.5}
Let $ i=j+m-1 $, the explicit compositional inverses of $ f(x)=x^{d_1}(1+(x^{2^m-1})^{-2^{j-1}}+(x^{2^m-1})^{2^{i-1}}) $ is
$$ f^{-1}(x)=(\alpha(x)h(g^{-1}(\beta(x)))^{2^m-2})^{r_{3}}(\beta(x)), $$
where $ \alpha(x)=x^{2^{2m}-2^m+1} $, $ \beta(x)=x^{2^m-1} $, $ h(x)=1+x^{-2^{j-1}}+x^{2^{i-1}} $, $ l(x)=x^{r_1r_2} $, $ r_1 \cdot 2^{m-1} \equiv 1 \pmod {2^m+1} $, $ r_2 \cdot 2^{j-1} \equiv 1 \pmod {2^m+1} $, $ r_3 \cdot d_1 \equiv 1 \pmod {2^{2m}-1} $.
\end{theorem}	
	
\section{Conclusion}
In this paper, we have proposed a new method to construct permutation trinomials with coefficients 1. Moreover, we have given the explicit compositional inverses of a class of permutation trinomials for a special case.
	
By using fractional permutation polynomials over $U$, we reversely deduced some PPs over the original finite fields. This construction method (or reversion usage of Tu's method) by polar decomposition has not appeared previously, and may be employed to discover more PPs (especially PPs with coefficients 1 since there were no simple characterization of these PPs over finite fields in the literature by now).
	
\section*{Acknowledgment}	
This work was supported by Natural Science Foundation of Beijing Municipality (No. 4202037), NSF of China with contract (No.61972018).

\end{document}